\documentclass[a4paper]{amsart}
\usepackage{color}
\usepackage{amssymb,amsmath,amsthm,amstext,amsfonts}
\usepackage[dvips]{graphicx}
\usepackage{psfrag}
\usepackage{url}
\usepackage{amsfonts}
\usepackage{amsmath}
\usepackage{mathrsfs}
\usepackage{graphicx}
\usepackage{amsmath,amsthm,amssymb,amscd}
\usepackage{enumerate}
\usepackage{circledsteps}
\usepackage[colorlinks,linkcolor=black,anchorcolor=black,citecolor=black,hyperindex=true,CJKbookmarks=true]{hyperref}

\usepackage{epsfig}

\makeatletter \@addtoreset{equation}{section} \makeatother

\renewcommand\thetable{\thesection.\@arabic\c@table}

\theoremstyle{plain}
\newtheorem{maintheorem}{Theorem}

\newtheorem{maincorollary}{Corollary}

\newtheorem{theorem}{Theorem}[section]

\newtheorem{lemma}{Lemma}[section]

\newtheorem{definition}{Definition}[section]
\newtheorem{remark}{Remark}[section]
\newtheorem{example}{Example}[section]

\theoremstyle{remark}

\long\def\begcom#1\endcom{}

\newcommand{\length}{\operatorname{\length}}

\def\length{\operatorname{length}}

\newcommand{\bl} {\begin{lemma}}
\newcommand{\el} {\end{lemma}}
\newcommand{\bt} {\begin{theorem}}
\newcommand{\et} {\end{theorem}}

\newcommand{\bp}{\begin{proof}}
\newcommand{\ep}{\end{proof}}
\newcommand  {\ee} {\end{equation}}
\newcommand  {\beq} {\begin{eqnarray*}}
\newcommand  {\eeq} {\end{eqnarray*}}

\newcommand  {\bd} {\begin{definition}}
\newcommand  {\ed} {\end{definition}}

\newcommand{\diam}{\operatorname{diam}}

\newcommand{\cM}{\mathcal{M}}

\def\ep{\noindent{\hfill $\Box$}}

\makeatletter
\@namedef{subjclassname@2020}{2020 Mathematics Subject Classification}
\makeatother

\begin{document}
\setlength{\footskip}{50pt}

\title{Uniqueness of ergodic optimization of top Lyapunov exponent for typical matrix cocycles}

\author{Wanshan Lin and Xueting Tian}

\address{Wanshan Lin, School of Mathematical Sciences,  Fudan University\\Shanghai 200433, People's Republic of China}
\email{21110180014@m.fudan.edu.cn}

\address{Xueting Tian, School of Mathematical Sciences,  Fudan University\\Shanghai 200433, People's Republic of China}
\email{xuetingtian@fudan.edu.cn}

\begin{abstract}
In this article, we consider the ergodic optimization of the top Lyapunov exponent. We prove that there is a unique maximising measure of top Lyapunov expoent for typical matrix cocyles. By using the results we obtain, we prove that in any non-uniquely ergodic minimal dynamical system, the Lyapunov-irregular points are typical for typical matrix cocyles.
\end{abstract}

\maketitle

\section{Introduction}

Let $(X,d)$ be a compact metric space, and $T:X \rightarrow X$ be a continuous map. Such $(X,T)$ is called a dynamical system. Let $\cM(X)$, $\cM(X,T)$, $\cM^{e}(X,T)$ denote the spaces of probability measures, $T$-invariant, $T$-ergodic probability measures, respectively. Let $\mathbb{Z}$, $\mathbb{N}$, $\mathbb{N^{+}}$ denote integers, non-negative integers, positive integers, respectively. Let $C(X)$ denote the space of real continuous functions on $X$ with the norm $\|f\|:=\sup\limits_{x\in X}|f(x)|.$ A sequence $\{f_n\}_{n=1}^{+\infty}$ from $X$ to $\mathbb{R}\cup\{-\infty\}$ will be called \emph{subadditive} if, for any $x\in X$ and any $n,m\in\mathbb{N^+}$, the inequality $f_{n+m}(x)\leq f_n(T^mx)+f_m(X)$ is satisfied. In a Baire space, a set is said to be residual if it has a dense $G_\delta$ subset.

For any $f\in C(X)$, we define the \emph{maximum ergodic average} $\beta(f):=\sup\limits_{\mu\in\cM(X,T)}\int f\mathrm{d\mu}.$ And the set of all \emph{maximising measures} $\cM_{max}(f):=\left\{\mu\in\cM(X,T)\mid \int f\mathrm{d\mu}=\beta(f)\right\}.$ When $f\in C(X)$, the study of the functional $\beta$ and $\cM_{max}(f)$ has been termed \emph{ergodic optimisation of Birkhoff averages}, and has attracted some research interest, see \cite{B2018,J2006,J2006ii,M2010} for more information.

In this article, we will pay attention to the ergodic optimization of the top Lyapunov exponent. Let $C(X,GL_d(\mathbb{R}))$ denote the space of all continuous functions $X\to GL_d(\mathbb{R})$. For continuous functions $A,B\in C(X,GL_d(\mathbb{R}))$, we use the metric $\rho(A,B):=\max\limits_{x\in X}\{\|A(x)-B(x)\|+\|A(x)^{-1}-B(x)^{-1}\|\}$, which makes $C(X,GL_d(\mathbb{R}))$ a complete metric space.

For any $A\in C(X,GL_d(\mathbb{R}))$, any $n\in\mathbb{N^+}$, a \emph{cocycle} $A(n,x)$ is defined as $A(n,x)=A(T^{n-1}x)\cdots A(x)$. Then $\{log\|A(n,x)\|\}_{n=1}^{+\infty}$ is a subadditive sequence of $C(X)$. For any $\mu\in\cM(X,T)$, by Kingman subadditive ergodic theorem, we have that $\lim\limits_{n\to+\infty}\frac{log\|A(n,x)\|}{n}$ exists and is called the \emph{top Lyapunov exponent} at $x$ for $\mu$-a.e. $x$. Moreover, there is a $T$-invariant function $\phi(x)$ such that $\phi(x)=\lim\limits_{n\to+\infty}\frac{log\|A(n,x)\|}{n}$ for $\mu$-a.e. $x$. and $\int \phi(x)\mathrm{d\mu}=\lim\limits_{n\to+\infty}\frac{1}{n}\int log\|A(n,x)\|\mathrm{d\mu}=\inf\limits_{n\geq1}\frac{1}{n}\int log\|A(n,x)\|\mathrm{d\mu}$. We define the \emph{maximum ergodic average} of $A$ to be the quantity $$\beta(A):=\sup_{\mu\in\cM(X,T)}\inf\limits_{n\geq1}\frac{1}{n}\int log\|A(n,x)\|\mathrm{d\mu}.$$ We shall say that $\mu\in\cM(X,T)$ is a \emph{maximising measure} for $A$ if $\inf\limits_{n\geq1}\frac{1}{n}\int log\|A(n,x)\|\mathrm{d\mu}=\beta(A)$, and denote the set of all maximising measures by $\cM_{max}(A)$.

Generally, for a subadditive sequence $\{f_n\}_{n=1}^{+\infty}$ of upper semicontimuous functions from $X$ to $\mathbb{R}\cup\{-\infty\}$, we can define the \emph{maximum ergodic average} $$\beta[(f_n)]:=\sup_{\mu\in\cM(X,T)}\inf\limits_{n\geq1}\frac{1}{n}\int f_n\mathrm{d\mu}.$$ And the set of all \emph{maximising measures} $$\cM_{max}[(f_n)]:=\left\{\mu\in\cM(X,T)\mid \inf\limits_{n\geq1}\frac{1}{n}\int f_n\mathrm{d\mu}=\beta[(f_n)]\right\}.$$ 

As for the set $\cM_{max}[(f_n)]$ and $\beta[(f_n)]$, Ian D. Morris proved
\begin{theorem}\cite[Proposition A.5]{M2013}\label{the1.1}
	Suppose that $\{f_n\}_{n=1}^{+\infty}$ is a subadditive sequence of upper semicontimuous functions from $X$ to $\mathbb{R}\cup\{-\infty\}$. Then $\cM_{max}[(f_n)]$ is compact, convex and nonempty, and the extreme points of $\cM_{max}[(f_n)]$ are precisely the ergodic elements of $\cM_{max}[(f_n)]$.
\end{theorem}
\begin{theorem}\cite[Theorem A.3]{M2013}\label{the1.2}
	Suppose that $\{f_n\}_{n=1}^{+\infty}$ is a subadditive sequence of $C(X)$. Then,
	\[
	\begin{split}
		\beta[(f_n)]&=\sup_{\mu\in\cM^{e}(X,T)}\inf\limits_{n\geq1}\frac{1}{n}\int f_n\mathrm{d\mu}=\inf\limits_{n\geq1}\sup_{x\in X}\frac{1}{n}f_n(x)\\&=\inf\limits_{n\geq1}\frac{1}{n}\sup_{\mu\in\cM(X,T)}\int f_n\mathrm{d\mu}=\sup_{x\in X}\inf\limits_{n\geq1}\frac{1}{n}f_n(x).
	\end{split}
	\]
	In all but the last of these expressions, the infimum over all $n\geq1$ may be replaced with the limit as $n\to+\infty$ of the same quantity, without altering the value of the expression. Furthermore, every supremum arising in each of the above expressions is attained.
\end{theorem}

Now, Support that $\Lambda$ is a nonempty and compact subset of $\cM(X,T)$. Similarly, for any $f\in C(X)$, any $A\in C(X,GL_d(\mathbb{R}))$, any subadditive sequence $\{f_n\}_{n=1}^{+\infty}$ of upper semicontimuous functions from $X$ to $\mathbb{R}\cup\{-\infty\}$. We define 
\[
\begin{split}
\beta^\Lambda(f)&:=\sup\limits_{\mu\in\Lambda}\int f\mathrm{d\mu},\\
\cM^\Lambda_{max}(f)&:=\left\{\mu\in\Lambda\mid \int f\mathrm{d\mu}=\beta^\Lambda(f)\right\},\\
\beta^\Lambda(A)&:=\sup_{\mu\in\Lambda}\inf\limits_{n\geq1}\frac{1}{n}\int log\|A(n,x)\|\mathrm{d\mu},\\
\cM^\Lambda_{max}(A)&:=\left\{\mu\in\Lambda\mid \inf\limits_{n\geq1}\frac{1}{n}\int log\|A(n,x)\|\mathrm{d\mu}=\beta^\Lambda(A)\right\},\\
\mathcal{U}^\Lambda&:=\left\{A\in C(X,GL_d(\mathbb{R}))\mid \cM^\Lambda_{max}(A) \text{ is a singleton}\right\},\\
\beta^\Lambda[(f_n)]&:=\sup_{\mu\in\Lambda}\inf\limits_{n\geq1}\frac{1}{n}\int f_n\mathrm{d\mu},\\
\cM^\Lambda_{max}[(f_n)]&:=\left\{\mu\in\Lambda\mid \inf\limits_{n\geq1}\frac{1}{n}\int f_n\mathrm{d\mu}=\beta^\Lambda[(f_n)]\right\}.
\end{split}
\]

From the definitions, if we denote $f_n=\sum\limits_{i=0}^{n-1}f\circ T^i$ and $A=e^fI_d$, then $\beta^\Lambda[(f_n)]=\beta^\Lambda(A)=\beta^\Lambda(f)$ and $\cM^\Lambda_{max}[(f_n)]=\cM^\Lambda_{max}(A)=\cM^\Lambda_{max}(f)$. For any $A\in\mathcal{U}^\Lambda$, suppose that $\cM_{max}(A)=\{\mu^\Lambda_A\}$. When $\Lambda=\cM(X,T)$, we will omit $\Lambda$. Since $\mu\to\inf\limits_{n\geq1}\frac{1}{n}\int log\|A(n,x)\|\mathrm{d\mu}$ is a upper semicontinuous function from $\Lambda$ to $\mathbb{R}$, we have that $\cM^\Lambda_{max}(A)$ is a nonempty and compact subset of $\cM(X,T)$. Similarly, $\cM^\Lambda_{max}(f)$ and $\cM^\Lambda_{max}[(f_n)]$ are nonempty and compact subsets of $\cM(X,T)$. 

As for $\mathcal{U}^\Lambda$, we will prove the following conclusions.
\begin{maintheorem}\label{maintheorem-1}
	Suppose that $(X,T)$ is a dynamical system, $\Lambda$ is a nonempty and compact subset of $\cM(X,T)$. Then
	\begin{enumerate}[(1)]
		\item $\mathcal{U}^\Lambda$ is residual in $C(X,GL_d(\mathbb{R}))$, in particular, $\mathcal{U}$ is residual in $C(X,GL_d(\mathbb{R}))$;
		\item if $\#\Lambda<+\infty$, then $C(X,GL_d(\mathbb{R}))\setminus\mathcal{U}^\Lambda$ is nowhere dense in $C(X,GL_d(\mathbb{R}))$;
		\item if $\overline{\Lambda\cap\cM^e(X,T)}=\Lambda$ and $\mathcal{F}$ is a dense subset of $C(X,GL_d(\mathbb{R}))$, then $\bigcup\limits_{A\in\mathcal{F}}\cM^\Lambda_{max}(A)$ is dense in $\Lambda$.
	\end{enumerate}
\end{maintheorem}

The particular case of Theorem \ref{maintheorem-1} (1) is a generalization of \cite[Theorem 3.2]{J2006} from continuous functions to matrix cocycles. As corollaries, we have 
\begin{maincorollary}\label{maincorollary-1}
	Suppose that $(X,T)$ is a dynamical system, if $\mathcal{F}$ is a dense subset of $C(X,GL_d(\mathbb{R}))$, then $\cM^e(X,T)\cap\bigcup\limits_{A\in\mathcal{F}}\cM_{max}(A)$ is dense in $\cM^e(X,T)$.
\end{maincorollary}
\begin{maincorollary}\label{maincorollary-2}
	Suppose that $(X,T)$ is a dynamical system with $\#\cM^e(X,T)<+\infty$, then $C(X,GL_d(\mathbb{R}))\setminus\mathcal{U}$ is nowhere dense in $C(X,GL_d(\mathbb{R}))$.
\end{maincorollary}
Here is an example such that $\#\cM^e(X,T)=+\infty$ and $C(X,GL_d(\mathbb{R}))\setminus\mathcal{U}$ is dense in $C(X,GL_d(\mathbb{R}))$ for some $d\geq1$.
\begin{example}
	Suppose that $X=[0,1]$ and $Tx=x$ for any $x\in[0,1]$. Then $\delta_x\in\cM^e(X,T)$ for any $x\in[0,1]$. Given $f\in C(X,\mathbb{R}\setminus\{0\})$. For any $n\geq1$, let 
	\[g_n(x):=
	\begin{cases}
		\frac{2^n-1}{2^n}\|f\|&\text{if } x\in f^{-1}([\frac{2^n-1}{2^n}\|f\|,\|f\|]),\\
		-\frac{2^n-1}{2^n}\|f\|&\text{if } x\in f^{-1}([-\|f\|,-\frac{2^n-1}{2^n}\|f\|]),\\
		f(x)&\text{otherwise}.
	\end{cases}
	\]
	Then $\|g_n-f\|\leq\frac{1}{2^n}\|f\|$ and $\#\cM_{max}(g_n)>1$ for any $n\geq1$. Hence, $C(X,\mathbb{R}\setminus\{0\})\setminus\mathcal{U}$ is dense in $C(X,\mathbb{R}\setminus\{0\})$.
\end{example}

For any $A\in C(X,GL_d(\mathbb{R}))$, we denote $$LI_A:=\left\{x\in X\mid \lim_{n\to+\infty}\frac{1}{n}log\|A(n,x)\|\text{ diverges }\right\},$$ the set of \emph{Lyapunov-irregular} points of $A$. Denote $$\mathcal{R}:=\left\{A\in C(X,GL_d(\mathbb{R}))\mid LI_A\text{ is residual in } X\right\}.$$ 

As for a minimal dynamical system, we have
\begin{maincorollary}\label{maincorollary-3}
	Suppose that $(X,T)$ is a minimal dynamical system with $\#\cM^e(X,T)>1$, then $\mathcal{R}$ is residual in $C(X,GL_d(\mathbb{R}))$. Moreover, if $1<\#\cM^e(X,T)<+\infty$, then $C(X,GL_d(\mathbb{R}))\setminus\mathcal{R}$ is nowhere dense in $C(X,GL_d(\mathbb{R}))$.
\end{maincorollary}
The rest of this paper is organized as follows. In section \ref{sec2}, we will introduce some preliminary results that will be used in the proof. In section \ref{sec3}, we will prove Theorem \ref{maintheorem-1}. In section \ref{sec4}, we will prove Corollary \ref{maincorollary-1}, \ref{maincorollary-2} and \ref{maincorollary-3}.

\section{Preliminaries}\label{sec2}
In this section, we will introduce some preliminary results that will be used in the proof.
\subsection{Subadditive sequence}
As for a subadditive sequence of upper semicontimuous functions from $X$ to $\mathbb{R}\cup\{-\infty\}$, it's be proved in \cite{CFH2008} that
\begin{theorem}\cite[Lemma 2.3]{CFH2008}\label{lem5}
	Suppose that $\{\nu_n\}_{n=1}^{+\infty}$ is a sequence in $\cM(X)$ and $\{f_n\}_{n=1}^{+\infty}$ is a subadditive sequence of upper semicontimuous functions from $X$ to $\mathbb{R}\cup\{-\infty\}$. We form the new sequence $\{\mu_{n}\}_{n=1}^{+\infty}$ by $\mu_{n}=\frac{1}{n}\sum\limits_{i=0}^{n-1}\nu_n\circ T^{-i}$. Assume that $\mu_{n_i}$ converges to $\mu$ in $\cM(X)$ for some subsequence $\{n_i\}$ of natural numbers. Then $\mu\in\cM(X,T)$, and moreover $$ \limsup_{i\to+\infty}\frac{1}{n_i}\int f_{n_i}\mathrm{d}\nu_{n_i}\leq\lim\limits_{n\to+\infty}\frac{1}{n}\int f_n\mathrm{d\mu}.$$
\end{theorem}
\subsection{Ergodic optimisation of Birkhoff averages}
In the research of ergodic optimisation of Birkhoff averages, it's be proved in \cite{J2006ii} that
\begin{theorem}\cite[Theorem 1]{J2006ii}\label{lem8}
	For any $\mu\in\cM^e(X,T)$, there exists $h\in C(X)$ such that $\cM_{max}(h)=\{\mu\}$.
\end{theorem}
\begin{remark}
	Since $\mu\mapsto\inf\limits_{n\geq1}\frac{1}{n}\int log\|A(n,x)\|\mathrm{d\mu}=\lim\limits_{n\to+\infty}\frac{1}{n}\int log\|A(n,x)\|\mathrm{d\mu}$ is an affine upper semicontinuous function on $\cM(X,T)$. If we combine \cite[Lemma 2(ii)]{J2006ii}, \cite[Proposition 1]{J2006ii} and \cite[Proposition 5]{Phe1}, then we can obtain a more general result: for any $A\in C(X,GL_d(\mathbb{R}))$ and any $\mu\in\cM^e(X,T)$, there is $f\in C(X)$ such that $\cM_{max}(e^fA)=\{\mu\}$ and $\beta(e^fA)=0$.
\end{remark}
\subsection{Lyapunov-irregular set} 
As for a minimal dynamical system, it's be proved in \cite{HLT2021} that
\begin{theorem}\cite[Corollary 1.2]{HLT2021}\label{the2.3}
	Suppose that $(X,T)$ is a minimal dynamical system. Given $A\in C(X,GL_d(\mathbb{R}))$, then either there is $c\in\mathbb{R}$ such that $\lim\limits_{n\to+\infty}\frac{1}{n}log\|A(n,x)\|=c$ for any $x\in X$, or $A\in\mathcal{R}$. 
\end{theorem}

\section{Proof of Theorem \ref{maintheorem-1}}\label{sec3}
In general, the proof of Theorem \ref{maintheorem-1} (1) is inspired by the proof of \cite[Theorem 3.2]{J2006}. In which, Oliver Jenkinson proved an analogous result for $\cM_{max}(f)$ with $f\in C(X)$. The proof of Theorem \ref{maintheorem-1} (2) is based on Theorem \ref{maintheorem-1} (1). The key to the proof of Theorem \ref{maintheorem-1} (3) is Theorem \ref{lem8} \cite[Theorem 1]{J2006ii}.

\subsection{Proof of Theorem \ref{maintheorem-1} (1)}
\begin{lemma}\label{lem1}
	Support that $\Lambda$ is a nonempty and compact subset of $\cM(X,T)$ and $\{f_n\}_{n=1}^{+\infty}$ is a subadditive sequence of upper semicontimuous functions from $X$ to $\mathbb{R}\cup\{-\infty\}$. Then 
		$$\beta^\Lambda[(f_n)]=\inf\limits_{n\geq1}\frac{1}{n}\sup_{\mu\in\Lambda}\int f_n\mathrm{d\mu}
		=\lim_{n\to+\infty}\frac{1}{n}\sup_{\mu\in\Lambda}\int f_n\mathrm{d\mu}.$$
\end{lemma}
\begin{proof}
	For the first equality, it's enough to show that $\beta^\Lambda[(f_n)]\geq\inf\limits_{n\geq1}\frac{1}{n}\sup\limits_{\mu\in\Lambda}\int f_n\mathrm{d\mu}$. The second equality is from the fact that $\sup\limits_{\mu\in\Lambda}\int f_{m+n}\mathrm{d\mu}\leq\sup\limits_{\mu\in\Lambda}\int f_m\mathrm{d\mu}+\sup\limits_{\mu\in\Lambda}\int f_n\mathrm{d\mu}$ for any $n\geq1, m\geq1$.
	
	For each $n\geq1$, we choose $\nu_n\in\cM^\Lambda_{max}(f_n)$. Since $\Lambda$ is compact, it's closed. Suppose that $\mu$ is a limit point of $\{\nu_{n}\}$, then $\mu\in\Lambda$. By Theorem \ref{lem5}, we have
	\[
	\begin{split}
		\inf_{n\geq1}\frac{1}{n}\int f_n\mathrm{d\mu}&=\lim\limits_{n\to+\infty}\frac{1}{n}\int f_n\mathrm{d\mu}\\&\geq\limsup_{i\to+\infty}\frac{1}{n_i}\int f_{n_i}\mathrm{d}\nu_{n_i}\\&=\limsup_{i\to+\infty}\frac{1}{n_i}\sup_{\nu\in\Lambda}\int f_{n_i}\mathrm{d}\nu\\&=\lim_{i\to+\infty}\frac{1}{i}\sup_{\nu\in\Lambda}\int f_{i}\mathrm{d}\nu\\&\geq\beta^\Lambda[(f_n)].
	\end{split}
	\]
	Hence, $\mu\in\cM^\Lambda_{max}[(f_n)]$ and $\beta^\Lambda[(f_n)]\geq\inf\limits_{n\geq1}\frac{1}{n}\sup\limits_{\mu\in\Lambda}\int f_n\mathrm{d\mu}$.
\end{proof}
From the proof of this lemma, we have
\begin{theorem}
	Support that $\Lambda$ is a nonempty and compact subset of $\cM(X,T)$ and $\{f_n\}_{n=1}^{+\infty}$ is a subadditive sequence of upper semicontimuous functions from $X$ to $\mathbb{R}\cup\{-\infty\}$, we choose $\mu_{n}\in\cM^\Lambda_{max}(f_n)$ for each $n\geq1$. If $\mu$ is a limit point of $\{\mu_{n}\}$, then $\mu\in\cM^\Lambda_{max}[(f_n)]$.
\end{theorem}

Since for any fixed $n\geq1$, $\sup\limits_{\mu\in\Lambda}\frac{1}{n}\int log\|A(n,x)\|\mathrm{d\mu}$ is a continuous function from $C(X,GL_d(\mathbb{R}))$ to $\mathbb{R}$. By Lemma \ref{lem1}, we have $\beta^\Lambda$ is the limit of an everywhere convergent sequence of continuous functions. By \emph{Baire's theorem on functions of first class} \cite[Theorem 7.3]{Oxtoby1980}, if we denote $$\tilde{\mathcal{C}^\Lambda}:=\left\{A\in C(X,GL_d(\mathbb{R}))\mid\beta^\Lambda \text{ is continuous at }A\right\},$$ then $C(X,GL_d(\mathbb{R}))\setminus\tilde{\mathcal{C}^\Lambda}$ is first category, which means that $\tilde{\mathcal{C}}^\Lambda$ is residual in $C(X,GL_d(\mathbb{R}))$.
\begin{remark}
	$\tilde{\mathcal{C}^\Lambda}$ may not be $C(X,GL_d(\mathbb{R}))$, see \cite[Corollary 6]{F1997} for an example.
\end{remark} 
For any $f\in C(X)$, we define a function $\Gamma(f):C(X,GL_d(\mathbb{R}))\to C(X,GL_d(\mathbb{R}))$ as $\Gamma(f)(A):=e^fA$ for any $A\in C(X,GL_d(\mathbb{R}))$. Then $\Gamma(f)$ is a topological homeomorphism with $(\Gamma(f))^{-1}=\Gamma(-f)$. Generally, for $n\in\mathbb{Z}$, 
\[(\Gamma(f))^n:=
\begin{cases}
	\underbrace{\Gamma(f)\circ\cdots\circ\Gamma(f)}_\text{n terms}&\text{if } n>0,\\
	id&\text{if } n=0,\\
	\underbrace{(\Gamma(f))^{-1}\circ\cdots\circ(\Gamma(f))^{-1}}_\text{(-n) terms}&\text{if } n<0.
\end{cases}
\]
Then, we have that $(\Gamma(f))^n=\Gamma(nf)$, for any $n\in\mathbb{Z}$. 

By \cite[Theorem 6.4]{Walters1982} there exists a countable  set of continuous functions $\{\gamma_i\}_{i=1}^{+\infty}$  such that $\|\gamma_i\|=1$ for any $i\in\mathbb{N^+}$ and  $$\varrho(\mu_{1},\mu_{2}):=\sum_{i=1}^{+\infty}\frac{|\int \gamma_i\mathrm{d}\mu_{1}-\int \gamma_i\mathrm{d}\mu_{2}|}{2^{i}}$$ defines a metric for the weak*-topology on $\mathcal{M}(X).$

Suppose that $\{\frac{1}{j}\gamma_i\mid i\in\mathbb{N^+},j\in\mathbb{N^+}\}=\{\eta_i\mid i\in\mathbb{N^+}\}$. We define a sequence subsets $\{\mathcal{C}^\Lambda_n\}_{n=1}^{+\infty}$ of $\tilde{\mathcal{C}^\Lambda}$ by induction: $\mathcal{C}^\Lambda_1:=\bigcap\limits_{i\in\mathbb{Z}}\Gamma(i\eta_1)(\tilde{\mathcal{C}^\Lambda})$, $\mathcal{C}^\Lambda_{n+1}:=\bigcap\limits_{i\in\mathbb{Z}}\Gamma(i\eta_{n+1})(\mathcal{C}^\Lambda_n)\subset\mathcal{C}^\Lambda_n\subset\tilde{\mathcal{C}^\Lambda}$, for any $n\geq1$. Since $\tilde{\mathcal{C}^\Lambda}$ is residual in $C(X,GL_d(\mathbb{R}))$, we have that $\mathcal{C}^\Lambda_n$ is residual in $C(X,GL_d(\mathbb{R}))$, for any $n\geq1$. Let $$\mathcal{C}^\Lambda:=\bigcap\limits_{n=1}^{+\infty}\mathcal{C}^\Lambda_n,$$ then it can be checked that
\begin{enumerate}[(1)]
	\item $\mathcal{C}^\Lambda\subset\tilde{\mathcal{C}^\Lambda}$;
	\item $\mathcal{C}^\Lambda$ is residual in $C(X,GL_d(\mathbb{R}))$;
	\item $\Gamma(\eta_n)(\mathcal{C}^\Lambda)=\Gamma(-\eta_n)(\mathcal{C}^\Lambda)=\mathcal{C}^\Lambda$ for any $n\geq1$.
\end{enumerate}

\begin{lemma}\label{lem2}
	Suppose that $Y$ is a complete metric space, a set $E\subset Y$ satisfies that $E$ is residual in $Y$. If $F\subset E$ and $F$ is residual in $E$, then $F$ is residual in $Y$.
\end{lemma}
\begin{proof}
	Suppose that $E\supset\bigcap\limits_{n=1}^{+\infty}G_n$, where $G_n$ is an open and dense subset of $Y$ for each $n\geq1$. since $F$ is residual in $E$, we can find a sequence $\{H_n\}_{n=1}^{+\infty}$ such that $F\supset\bigcap\limits_{n=1}^{+\infty}H_n$ and $H_n$ is an open and dense subset of $E$ for each $n\geq1$. Hence, there is a sequence $\{I_n\}_{n=1}^{+\infty}$ such that $H_n=I_n\cap E$, and $I_n$ is an open and dense subset of $Y$ for each $n\geq1$. As a result, $$F\supset\bigcap\limits_{n=1}^{+\infty}H_n=\bigcap\limits_{n=1}^{+\infty}(I_n\cap E)\supset\bigcap\limits_{n=1}^{+\infty}\bigcap\limits_{j=1}^{+\infty}(I_n\cap G_j).$$ Then $F$ is residual in $Y$. 
\end{proof}
We denote $$\mathcal{U}^\Lambda(\mathcal{C}^\Lambda):=\left\{A\in \mathcal{C}^\Lambda\mid \cM^\Lambda_{max}(A) \text{ is a singleton}\right\}.$$ Then to prove Theorem \ref{maintheorem-1} (1), we only need to prove that $\mathcal{U}^\Lambda(\mathcal{C}^\Lambda)$ is residual in $\mathcal{C}^\Lambda$.

\begin{lemma}\label{lem3}
Suppose that $A\in C(X,GL_d(\mathbb{R}))$ and $\{A_n\}_{n=1}^{+\infty}$ is a sequence of $C(X,GL_d(\mathbb{R}))$. For each $n\geq1$, $\mu_{n}\in\cM^\Lambda_{max}(A_n)$ and $\mu\in\Lambda$ is a limit point of $\{\mu_{n}\}$. Suppose that
\begin{enumerate}[(1)]
\item $\lim\limits_{n\to+\infty}\rho(A_n,A)=0$;
\item $\lim\limits_{n\to+\infty}\beta^\Lambda(A_n)=\beta^\Lambda(A)$.
\end{enumerate}
Then $\mu\in\cM^\Lambda_{max}(A)$.
\end{lemma}
\begin{proof}
	Suppose that $\mu=\lim\limits_{k\to+\infty}\mu_{n_k}$. For any $m\geq1$, we have $$\lim_{k\to+\infty}\left|\frac{1}{m}\int\left(log\|A_{n_k}(m,x)\|-log\|A(m,x)\|\right)\mathrm{d}\mu_{n_k}\right|=0,$$ and $$\lim_{k\to+\infty}\left|\frac{1}{m}\int log\|A(m,x)\|\mathrm{d}\mu_{n_k}-\frac{1}{m}\int log\|A(m,x)\|\mathrm{d}\mu\right|=0.$$ Combining these two estimates, we obtain $$\lim_{k\to+\infty}\frac{1}{m}\int log\|A_{n_k}(m,x)\|\mathrm{d}\mu_{n_k}=\frac{1}{m}\int log\|A(m,x)\|\mathrm{d}\mu.$$ Hence, for any $m\geq1$, 	
	\[
	\begin{split}
		\frac{1}{m}\int log\|A(m,x)\|\mathrm{d}\mu&=\lim_{k\to+\infty}\frac{1}{m}\int log\|A_{n_k}(m,x)\|\mathrm{d}\mu_{n_k}\\
		&\geq\lim_{k\to+\infty}\inf\limits_{n\geq1}\frac{1}{n}\int log\|A_{n_k}(n,x)\|\mathrm{d}\mu_{n_k}\\&=\lim_{k\to+\infty}\beta^\Lambda(A_{n_k})\\&=\beta^\Lambda(A).
	\end{split}
	\]
Then $\inf\limits_{m\geq1}\frac{1}{m}\int log\|A(m,x)\|\mathrm{d}\mu\geq\beta^\Lambda(A)$ and $\mu\in\cM^\Lambda_{max}(A)$.
\end{proof}

Since for every $\varepsilon>0$, $n\geq1$, $$\sup_{\mu\in\Lambda}\left|\frac{1}{n}\int log\|e^{\varepsilon f}A(n,x)\|\mathrm{d\mu}-\frac{1}{n}\int log\|A(n,x)\|\mathrm{d\mu}\right|=\varepsilon\sup_{\mu\in\Lambda}\left|\int f\mathrm{d\mu}\right|\leq\varepsilon\|f\|.$$ Given $f\in C(X)$, we have $\lim\limits_{\varepsilon\to0}\beta^\Lambda(e^{\varepsilon f}A)=\beta^\Lambda(A)$ uniformly for $A\in C(X,GL_d(\mathbb{R}))$.
\begin{definition}
	Support that $\Lambda$ is a nonempty and compact subset of $\cM(X,T)$. For any $f\in C(X)$, $A\in C(X,GL_d(\mathbb{R}))$, define $$\beta^\Lambda(f\mid A):=\sup_{\mu\in\cM^\Lambda_{max}(A)}\int f\mathrm{d}\mu,$$ the relative maximum ergodic average of $f$ given $A$ and $\Lambda$. Define $$\cM^\Lambda_{max}(f\mid A):=\left\{\mu\in\cM^\Lambda_{max}(A)\mid\int f\mathrm{d}\mu=\beta^\Lambda(f\mid A)\right\}.$$
\end{definition}
Suppose that $C$ and $D$ are nonempty compact subsets of $\mathbb{R}$, then the \emph{Hausdorff metric} $d_H$ is defined by $d_H(C,D)=\max\limits_{a\in C}\min\limits_{b\in D}|a-b|+\max\limits_{c\in D}\min\limits_{d\in C}|c-d|$.
\begin{lemma}\label{lem4}
	Given $f\in C(X)$, $A\in C(X,GL_d(\mathbb{R}))$, we have $$\left\{\int f\mathrm{d\mu}\mid\mu\in\cM^\Lambda_{max}(e^{\varepsilon f}A)\right\}\longrightarrow\left\{\beta^\Lambda(f\mid A)\right\} \text{  as } \varepsilon\searrow 0,$$ in the sense of the Hausdorff metric.
\end{lemma}
\begin{proof}
	Given $f\in C(X)$, $A\in C(X,GL_d(\mathbb{R}))$. Since $\cM^\Lambda_{max}(e^{\varepsilon f}A)$ is a nonempty and compact subset of $\cM(X,T)$, the set $\left\{\int f\mathrm{d\mu}\mid\mu\in\cM^\Lambda_{max}(e^{\varepsilon f}A)\right\}$ is a nonempty compact subset of $\mathbb{R}$. To prove the lemma, it is enough to show that if $a_\varepsilon\in \left\{\int f\mathrm{d\mu}\mid\mu\in\cM^\Lambda_{max}(e^{\varepsilon f}A)\right\}$, then $\lim\limits_{\varepsilon \searrow 0}a_\varepsilon=\beta^\Lambda(f\mid A)$. Writing $a_\varepsilon=\int f\mathrm{d}m_\varepsilon$ for some $m_\varepsilon\in \cM^\Lambda_{max}(e^{\varepsilon f}A)$, it is in turn enough to prove that any limit point of $m_\varepsilon$, as $\varepsilon \searrow 0$, belongs to $\cM^\Lambda_{max}(f\mid A)$. 
	
	Suppose that $\lim\limits_{i\to+\infty}m_{\varepsilon_i}=m$. By Lemma \ref{lem3}, we have that $m\in\cM^\Lambda_{max}(A)$. For any fixed $i\geq1$, since $m_{\varepsilon_i}\in \cM^\Lambda_{max}(e^{\varepsilon_i f}A)$, we have $$\inf\limits_{n\geq1}\frac{1}{n}\int log\|A(n,x)\|\mathrm{d}m_{\varepsilon_i}+\varepsilon_i\int f\mathrm{d}m_{\varepsilon_i}\geq\inf\limits_{n\geq1}\frac{1}{n}\int log\|A(n,x)\|\mathrm{d\mu}+\varepsilon_i\int f\mathrm{d}\mu,$$ for any $\mu\in\Lambda$. For any fixed $\mu\in\cM^\Lambda_{max}(A)$, we have $$\inf\limits_{n\geq1}\frac{1}{n}\int log\|A(n,x)\|\mathrm{d\mu}\geq\inf\limits_{n\geq1}\frac{1}{n}\int log\|A(n,x)\|\mathrm{d}m_{\varepsilon_i},$$ for any $i\geq1$. Combining these two estimates, we obtain $\varepsilon_i\int f\mathrm{d}m_{\varepsilon_i}\geq\varepsilon_i\int f\mathrm{d}\mu$, i.e. $\int f\mathrm{d}m_{\varepsilon_i}\geq\int f\mathrm{d}\mu$ for any $\mu\in\cM^\Lambda_{max}(A)$, any $i\geq1$. Let $i\to+\infty$, then $\int f\mathrm{d}m\geq\int f\mathrm{d}\mu$ for any $\mu\in\cM^\Lambda_{max}(A)$. Hence, $m\in\cM^\Lambda_{max}(f\mid A)$.
\end{proof}
Let's finish the proof of Theorem \ref{maintheorem-1} (1).
\begin{proof}[Proof of Theorem \ref{maintheorem-1} (1).]
 For any $A\in\mathcal{C}^\Lambda$, any $i\geq1$, define $$M^\Lambda_i(A):=\left\{\int \gamma_i\mathrm{d\mu}\mid\mu\in\cM^\Lambda_{max}(A)\right\}.$$ Then $A\in\mathcal{U}^\Lambda(\mathcal{C}^\Lambda)$ if and only if $M^\Lambda_i(A)$ is a singleton for any $i\geq1$. Define $$E^\Lambda_{i,j}:=\left\{A\in\mathcal{C}^\Lambda\mid \diam(M^\Lambda_i(A))\geq\frac{1}{j}\right\},$$ where $i\geq1$, $j\geq1$. Then
	 \[
	 \begin{split}
	 	\mathcal{C}^\Lambda\setminus\mathcal{U}^\Lambda(\mathcal{C}^\Lambda)&=\left\{A\in\mathcal{C}^\Lambda\mid\diam(M^\Lambda_i(A))>0\text{ for some }i\in\mathbb{N^+}\right\}\\&=\bigcup_{i=1}^{+\infty}\bigcup_{j=1}^{+\infty}\left\{A\in\mathcal{C}^\Lambda\mid\diam(M^\Lambda_i(A))\geq\frac{1}{j}\right\}\\&=\bigcup_{i=1}^{+\infty}\bigcup_{j=1}^{+\infty}E^\Lambda_{i,j}.
	 \end{split}
	 \] And $$\mathcal{U}^\Lambda(\mathcal{C}^\Lambda)=\bigcap_{i=1}^{+\infty}\bigcap_{j=1}^{+\infty}(\mathcal{C}^\Lambda\setminus E^\Lambda_{i,j}).$$ 
	To prove Theorem \ref{maintheorem-1} (1), it is enough to show that $E^\Lambda_{i,j}$ is closed and has empty interior in $\mathcal{C}^\Lambda$.
	
	To show that $E^\Lambda_{i,j}$ is closed in $\mathcal{C}^\Lambda$, let $\{A_n\}_{n=1}^{+\infty}\subset E^\Lambda_{i,j}$ with $\lim\limits_{n\to+\infty}A_n=A\in\mathcal{C}^\Lambda$. Write $\diam(M^\Lambda_i(A_n))=\int\gamma_i\mathrm{d}\mu_{n}^+-\int\gamma_i\mathrm{d}\mu_{n}^-\geq\frac{1}{j}$ for measures $\mu_{n}^-,\mu_{n}^+\in\cM^\Lambda_{max}(A_n)$. Then there exists $\mu^-,\mu^+\in\Lambda$ such that $\lim\limits_{n\to+\infty}\mu_{r_n}^-=\mu^-$ and $\lim\limits_{n\to+\infty}\mu_{s_n}^+=\mu^+$, where $r_1<r_2<\cdots, s_1<s_2<\cdots$. By Lemma \ref{lem3}, $\mu^-,\mu^+\in\cM^\Lambda_{max}(A)$. Since $\int\gamma_i\mathrm{d}\mu_{r_n}^-\to\int\gamma_i\mathrm{d}\mu^-$ and $\int\gamma_i\mathrm{d}\mu_{s_n}^+\to\int\gamma_i\mathrm{d}\mu^+$, we have $\int\gamma_i\mathrm{d}\mu^+-\int\gamma_i\mathrm{d}\mu^-\geq\frac{1}{j}$. Hence, $A\in E^\Lambda_{i,j}$ and $E^\Lambda_{i,j}$ is closed.
	
	To show that $E^\Lambda_{i,j}$ has empty interior in $\mathcal{C}^\Lambda$, let $A\in E^\Lambda_{i,j}$ be arbitrary. By Lemma \ref{lem4}, $$M^\Lambda_i(e^{\varepsilon\gamma_i}A)=\left\{\int \gamma_i\mathrm{d\mu}\mid\mu\in\cM^\Lambda_{max}(e^{\varepsilon \gamma_i}A)\right\}\longrightarrow\left\{\beta^\Lambda(\gamma_i\mid A)\right\} \text{  as } \varepsilon\searrow 0.$$ In particular, $\diam(M^\Lambda_i(e^{\frac{1}{t}\gamma_i}A))<\frac{1}{j}$ for $t\in\mathbb{N^+}$ sufficiently large. Since $e^{\frac{1}{t}\gamma_i}A\in\mathcal{C}^\Lambda$ for each $t\in\mathbb{N^+}$, we have that $E^\Lambda_{i,j}$ has empty interior in $\mathcal{C}^\Lambda$.
\end{proof}
\subsection{Proof of Theorem \ref{maintheorem-1} (2)} When $\#\Lambda=1$, $\mathcal{U}^\Lambda=C(X,GL_d(\mathbb{R}))$. Now suppose that $\#\Lambda=m>1$ and $\Lambda=\{\mu_{1},\mu_{2},\cdots,\mu_{m}\}$. For any $1\leq i\leq m$, since $\beta(A,\mu_{i}):=\inf\limits_{n\geq1}\frac{1}{n}\int log\|A(m,x)\|\mathrm{d}\mu_{i}=\lim\limits_{n\to+\infty}\frac{1}{n}\int log\|A(m,x)\|\mathrm{d}\mu_{i}$ is the limit of an everywhere convergent sequence of continuous functions. By \emph{Baire's theorem on functions of first class} \cite[Theorem 7.3]{Oxtoby1980}, if we denote $$\tilde{\mathcal{C}^\Lambda_i}:=\left\{A\in C(X,GL_d(\mathbb{R}))\mid\beta(A,\mu_{i}) \text{ is continuous at }A\right\},$$ then $C(X,GL_d(\mathbb{R}))\setminus\tilde{\mathcal{C}^\Lambda_i}$ is first category, which means that $\tilde{\mathcal{C}^\Lambda_i}$ is residual in $C(X,GL_d(\mathbb{R}))$. Let $\mathcal{G}^\Lambda:=\mathcal{U}^\Lambda\cap\bigcap\limits_{i=1}^{m}\tilde{\mathcal{C}^\Lambda_i}$, then $\mathcal{G}^\Lambda$ is residual in $C(X,GL_d(\mathbb{R}))$.

Let's finish the proof of Theorem \ref{maintheorem-1} (2).
\begin{proof}[Proof of Theorem \ref{maintheorem-1} (2).]
	Given $A\in\mathcal{G}^\Lambda$ with $\cM_{max}^\Lambda(A)=\{\mu_{l}\}$. there is $k_l\in\mathbb{R}$ satisfying $\beta(A,\mu_{i})<k_l<\beta(A,\mu_{l})$ for any $i\neq l$. Since $A\in\bigcap\limits_{i=1}^{m}\tilde{\mathcal{C}^\Lambda_i}$, for any $1\leq j\leq m$, there is an open subset $U_j$ of $C(X,GL_d(\mathbb{R}))$ such that $A\in U_j$ and $|\beta(B,\mu_{j})-\beta(A,\mu_{j})|<\frac{1}{2}|k_l-\beta(A,\mu_{j})|$ for any $B\in U_j$. Let $U:=\bigcap\limits_{j=1}^{m}U_j$. Then $A\in U$ and $U$ is an open subset of $C(X,GL_d(\mathbb{R}))$. For any $B\in U$, we have $\beta(B,\mu_{i})<k_l<\beta(B,\mu_{l})$ for any $i\neq l$. Hence, $U\subset\mathcal{U}^\Lambda$. Since $\mathcal{G}^\Lambda$ is residual in $C(X,GL_d(\mathbb{R}))$, we have that $C(X,GL_d(\mathbb{R}))\setminus\mathcal{U}^\Lambda$ is nowhere dense in $C(X,GL_d(\mathbb{R}))$.
\end{proof}   
\subsection{Proof of Theorem \ref{maintheorem-1} (3)} 
\begin{lemma}\label{lem7}
	For any $f\in C(X)$, $e^fI_d\in\tilde{\mathcal{C}^\Lambda}$.
\end{lemma}
\begin{proof}
	(1) First, we will prove that $I_d\in\tilde{\mathcal{C}^\Lambda}$. Since $\|A(n,x)\|\leq\|A(T^{n-1}x)\|\cdots\|A(x)\|$ and $$\|A(n,x)\|\geq\frac{1}{\|A(n,x)^{-1}\|}=\frac{1}{\|A(x)^{-1}\cdots A(T^{n-1}x)^{-1}\|}\geq\frac{1}{\|A(x)^{-1}\|\cdots\|A(T^{n-1}x)^{-1}\|},$$ we have that $$-\int log\|A(x)^{-1}\|\mathrm{d\mu}\leq\frac{1}{n}\int log\|A(n,x)\|\mathrm{d\mu}\leq \int log\|A(x)\|\mathrm{d\mu},$$ for any $\mu\in\Lambda$ and any $n\geq1$. Therefore, $$|\beta^\Lambda(A)|\leq\sup_{\mu\in\Lambda}|\inf_{n\geq1}\frac{1}{n}\int log\|A(n,x)\|\mathrm{d\mu}|\leq\sup_{x\in X}\{|log\|A(x)\||,|log\|A(x)^{-1}\||\}.$$ Hence, $\beta^\Lambda(A)\to0$ as $A\to I_d$, which means that $I_d\in\tilde{\mathcal{C}^\Lambda}$.
	
	(2) Second, we will prove that $e^fI_d\in\tilde{\mathcal{C}}^\Lambda$ for any $f\in C(X)$. Suppose that $A_n\to e^fI_d$, then $e^{-f}A_n\to I_d$. Since $$\frac{1}{m}\int log\|A_n(m,x)\|\mathrm{d\mu}-\frac{1}{m}\int log\|(e^fI_d)(m,x)\|\mathrm{d\mu}=\frac{1}{m}\int log\|(e^{-f}A_n)(m,x)\|\mathrm{d\mu},$$ for any $\mu\in\Lambda$ and any $m\geq1$. Similarly, we have $\beta^\Lambda(A_n)\to\beta^\Lambda(e^fI_d)$ as $A_n\to e^fI_d$, which means that $e^fI_d\in\tilde{\mathcal{C}^\Lambda}$.
\end{proof}

Let's finish the proof of Theorem \ref{maintheorem-1} (3).
\begin{proof}[Proof of Theorem \ref{maintheorem-1} (3).]
    Suppose that $F:=\bigcup\limits_{A\in\mathcal{F}}\cM^\Lambda_{max}(A)$. Given $\mu\in\cM^e(X,T)\cap\Lambda$, by Theorem \ref{lem8}, there exists $h\in C(X)$ such that $\cM_{max}(h)=\{\mu\}$. Let $A=e^hI_d$, then $\cM^\Lambda_{max}(A)=\{\mu\}$ and $A\in\tilde{\mathcal{C}^\Lambda}\cap\mathcal{U}^\Lambda$ by Lemma \ref{lem7}. Since $\mathcal{F}$ is dense in $C(X,GL_d(\mathbb{R}))$, there exists a sequence $\{A_n\}_{n=1}^{+\infty}$ of $\mathcal{F}$ with $\lim\limits_{n\to+\infty}A_n=A$. Choose $\mu_{n}\in\cM^\Lambda_{max}(A_n)\subset F$ for each $n\geq1$, by Lemma \ref{lem3}, we have that $\lim\limits_{n\to+\infty}\mu_{n}=\mu$. Therefore, $F$ is dense in $\Lambda$.
\end{proof}

\section{Proofs of Corollary \ref{maincorollary-1}, \ref{maincorollary-2} and \ref{maincorollary-3}}\label{sec4}
In this section, we will finish the proof of Corollary \ref{maincorollary-1}, \ref{maincorollary-2} and \ref{maincorollary-3}.
\begin{proof}[Proof of Corollary \ref{maincorollary-1}]
	Suppose that $F:=\cM^e(X,T)\cap\bigcup\limits_{A\in\mathcal{F}}\cM_{max}(A)$. Given $\mu\in\cM^e(X,T)$, by Theorem \ref{lem8}, there exists $h\in C(X)$ such that $\cM_{max}(h)=\{\mu\}$. Let $A=e^hI_d$, then $\cM_{max}(A)=\{\mu\}$ and $A\in\tilde{\mathcal{C}}\cap\mathcal{U}$ by Lemma \ref{lem7}. Since $\mathcal{F}$ is dense in $C(X,GL_d(\mathbb{R}))$, there exists a sequence $\{A_n\}_{n=1}^{+\infty}$ of $\mathcal{F}$ with $\lim\limits_{n\to+\infty}A_n=A$. Choose $\mu_{n}\in\cM_{max}(A_n)\subset F$ for each $n\geq1$, by Lemma \ref{lem3}, we have that $\lim\limits_{n\to+\infty}\mu_{n}=\mu$. Therefore, $F$ is dense in $\cM^e(X,T)$.
\end{proof}
\begin{proof}[Proof of Corollary \ref{maincorollary-2}]
	Since $\#\cM^e(X,T)<+\infty$, let $\Lambda=\cM^e(X,T)$, then $\Lambda$ is a nonempty and compact subset of $\cM(X,T)$. By Theorem \ref{maintheorem-1} (2), $C(X,GL_d(\mathbb{R}))\setminus\mathcal{U}^\Lambda$ is nowhere dense in $C(X,GL_d(\mathbb{R}))$. Since $\mathcal{U}^\Lambda=\mathcal{U}$, we have that $C(X,GL_d(\mathbb{R}))\setminus\mathcal{U}$ is nowhere dense in $C(X,GL_d(\mathbb{R}))$.
\end{proof}
\begin{proof}[Proof of Corollary \ref{maincorollary-3}]
	When $\#\cM^e(X,T)>1$, by Theorem \ref{the2.3}, $\mathcal{U}\subset\mathcal{R}$. By Theorem \ref{maintheorem-1} (1), $\mathcal{R}$ is residual in $C(X,GL_d(\mathbb{R}))$. When $1<\#\cM^e(X,T)<+\infty$, by Corollary \ref{maincorollary-2}, $C(X,GL_d(\mathbb{R}))\setminus\mathcal{R}$ is nowhere dense in $C(X,GL_d(\mathbb{R}))$. 
\end{proof}

\bigskip

$\mathbf{Acknowledgements.}$ X. Tian is supported by National Natural Science Foundation of China (grant No.12071082) and in part by Shanghai Science and Technology Research Program (grant No.21JC1400700).

\end{document}